\documentclass[runningheads]{llncs}
\usepackage[T1]{fontenc}
\usepackage{amssymb}
\usepackage{amsmath}
\usepackage{bm}

\usepackage{lscape}
\usepackage{graphics}
\usepackage{times}
\usepackage{color}
\usepackage{colortbl}
\usepackage{caption}
\usepackage{epsfig,latexsym,graphicx}
\usepackage{psfrag}
\usepackage{comment}

\usepackage{amsmath,amscd,amssymb,amsfonts}
\usepackage{xspace}
\usepackage{listings}
\usepackage{subfig}
\usepackage{boxedminipage}
\usepackage{url}
\usepackage[ruled,vlined,linesnumbered]{algorithm2e}
\usepackage{multirow}
\usepackage{rotating}
\usepackage{wrapfig}
\usepackage{xcolor, soul}
\sethlcolor{green}
\newtheorem{observation}{Observation}
\begin{document}
\title{On the Span of $l$ Distance Coloring of Infinite Hexagonal Grid}
\author{Sasthi C. Ghosh \and Subhasis Koley}
\institute{Advanced Computing and Microelectronics unit\\ Indian Statistical Institute, 203 B. T. Road, Kolkata 700108, India\\ Emails: \email{sasthi@isical.ac.in, subhasis.koley2@gmail.com }\\
}
\authorrunning{Sasthi C. Ghosh et al.}
\maketitle              

\begin{abstract} For a graph $G(V,E)$ and $l \in \mathbb{N}$, an $l$ distance coloring is a coloring $f: V \to \{1, 2, \cdots, n\}$ of $V$ such that $\forall u,\;v \in V,\; u\neq v,\; f(u)\neq f(v)$ when  $d(u,v) \leq l$. Here $d(u,v)$ is the distance between $u$ and $v$ and is equal to the minimum number of edges that  connect $u$ and $v$ in $G$. The span of $l$ distance coloring of $G$, $\lambda ^{l}(G)$, is the minimum $n$ among all $l$ distance coloring of $G$. A class of channel assignment problem in cellular network can be formulated as a distance graph coloring problem in regular grid graphs. The cellular network is often modelled as an infinite hexagonal grid $T_H$, and hence determining $\lambda ^{l}(T_H)$ has relevance from practical point of view.  Jacko and Jendrol [Discussiones Mathematicae Graph Theory, $2005$] determined the exact value of $\lambda ^{l}(T_H)$ for any odd $l$ and for even $l \geq 8$, it is conjectured that  
$\lambda ^{l}(T_H) =  \left[  \dfrac{3}{8} \left( \, l+\dfrac{4}{3} \right) ^2 \right]$ where $[x]$ is an integer, $x\in \mathbb{R}$ and $x-\dfrac{1}{2} < [x] \leq x+\dfrac{1}{2}$. For $l=8$, the conjecture has been proved by Sasthi and Subhasis [$22$nd Italian Conference on Theoretical Computer Science, $2021$]. In this paper, we prove the conjecture for any $l \geq 10$. 

\end{abstract}

\section{Introduction}\label{sec:1}
In a wireless communication network, frequency channels are assigned to the transmitters for data communication. For interference free communication, frequency channels assigned to proximity transmitters or receivers must be distinct and must have predefined gaps among them. As the available frequency channel resources are limited, one of the objective of channel assignment problem (CAP) is to find out the least number of frequency channels needed for interference free communication. Frequently, cellular network is modelled as an infinite regular hexagonal grid $T_H$ or honeycomb grid (as $T_H$ is isomorphic to infinite regular honeycomb grid) for regular geometric pattern of $T_H$. In Fig.~\ref{honeycomb}, honeycomb representation of 
$T_H$ have been shown. The coordinates of the vertices of $T_H$ have also been shown here.
\begin{center}
\begin{figure}[!ht]
\centerline{\includegraphics[scale=1]{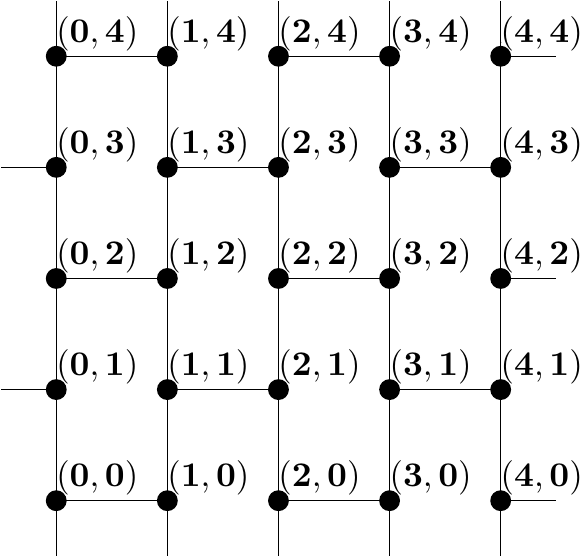}}
\caption{Honeycomb representation of $T_H$ and coordinates of its vertices.}
\label{honeycomb} 
\end{figure}
\end{center} 
The CAP can be modelled theoretically as a graph coloring problem where vertices of the graph represent receivers and color assigned to a vertex represents the frequency channel assigned to the corresponding receiver for communication~\cite{Hale}. Sometimes CAP is also modelled as an $l$ distance graph coloring problem~\cite{Wegner} or distance graph coloring problem to   incorporate the effect of interference up to multi hop distance into the modelling. Formally the distance between two vertices and the $l$ distance coloring are defined as follows.

\begin{definition}\label{defold1}~\cite{ICTCSconj}
For a graph $G(V,E)$, for any two distinct vertices $u,\;v\in V$, $d(u,v)$, the distance between $u$ and $v$, is defined as the minimum number of edges required to connect $u$ and $v$ in $G$.
\end{definition}

\begin{definition}\label{defold4}~\cite{ICTCSconj}
Given $l \in \mathbb{N}$, an $l$ distance coloring of $G$ is a coloring $f: V \to \{1, 2, \cdots, n\}$ of the vertices of $G$ such that $f(u)\neq f(v)$ for any two  distinct vertices $u$ and $v$ in $G$ with $d(u,v)\leq l$.
\end{definition}

The span $\lambda ^{l}(G)$ of $l$ distance coloring of $G$ is the minimum $n$ such that $G$ admits an $l$ distance coloring of $G$. As the cellular network is often modelled as $T_H$, several authors studied the distance graph coloring problem in $T_H$~\cite{Bertossi1,Bertossi2,conjecturejacko,soumenconj}. For odd $l$, the exact value of $\lambda ^l(T_H)$ has been determined in \cite{conjecturejacko}. For every even $l\geq 8$, it was conjectured in~\cite{conjecturejacko} that, 
\begin{eqnarray}\label{eq1}
\lambda ^l(T_H) =  \left[  \dfrac{3}{8} \left( \, l+\dfrac{4}{3} \right) ^2 \right]
\end{eqnarray} ($[x]$ is an integer, $x\in 
\mathbb{R}$ and $x-\dfrac{1}{2} < [x] \leq x+\dfrac{1}{2}$). Recently, Sasthi C. Ghosh and Subhasis Koley proved the conjecture for $l=8$~\cite{ICTCSconj}. But $\lambda ^l(T_H)$ has not been determined yet for even $l\geq 10$ \cite{conjecturejacko,surveykramer,spacking1,Spacking2}. In this paper, we prove the conjecture for $l\geq 10$. In next section, relevant preliminaries and definitions have been stated before going to the result section.

\section{Preliminaries}
\begin{definition}\label{defold2}~\cite{ICTCSconj}
A vertex with coordinates $(i,j)$ in $T_H$ is said to be a right vertex, or $x_r$ if it is connected to the vertex with coordinates $(i+1,j)$ by an edge. 
\end{definition}
A right vertex $x_r$ with coordinates $(i,j)$ is adjacent to the vertices having coordinates $(i+1,j)$, $(i,j+1)$ and $(i,j-1)$ but not adjacent to the vertex with coordinates $(i-1,j)$.

\begin{definition}\label{defold3}~\cite{ICTCSconj}
A vertex with coordinates $(i,j)$ in $T_H$ is said to be a left vertex, or $x_l$ if it is connected to the vertex with coordinates $(i-1,j)$ by an edge. 
\end{definition}
A left vertex $x_l$ with coordinates $(i,j)$ is adjacent to the vertices having coordinates $(i-1,j)$, $(i,j+1)$ and $(i,j-1)$ but not adjacent to the vertex with coordinates $(i+1,j)$.

\begin{definition}\label{defold5}~\cite{ICTCSconj}
A distance $2p$ clique $D^{2p}_x$ ($p \in \mathbb{N}$) of $T_H$ centered at vertex $x\in V(T_H)$ is defined as the maximum vertex induced sub-graph of $T_H$ such that for each pair of vertices $u,v \in V(D^{2p}_x)$, $d(u,v)\leq 2p$ and for every vertex $w \in V(D^{2p}_x)$, $d(w,x) \leq p$.    
\end{definition}
In Fig.~\ref{evenclique} different $D^{2p}_x$ in $T_H$ are shown. 
\begin{center}
\begin{figure}[h!]
\centerline{\includegraphics[scale=.8]{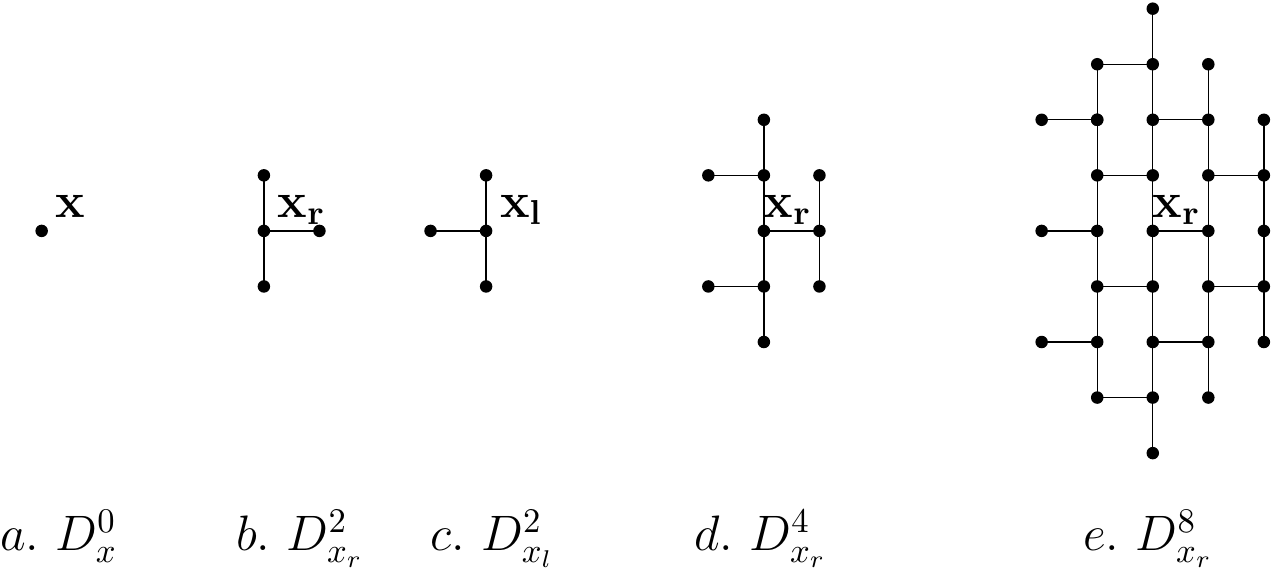}}
\caption{Different $D^{2p}_x$. }
\label{evenclique} 
\end{figure} 
\end{center}

Let us assume the right vertex $x_c(0,0)$ is the origin. As discussed in~\cite{conjecturejacko}, $T_H$ is a bipartite graph, where there exists two disjoint sets $V_0, V_1$ such that $V_0\cup V_1 =V(T_H)$ and for each edge $uv\in E(T_H)$, $u\in V_0$, $v\in V_1$ or otherwise. For any vertex $v(i,j)\in T_H$, $\tau(v)$ is defined as $\tau(v)=((i+j)\mod 2)$~\cite{conjecturejacko}. For a vertex $v(i,j)\in T_H$, if $\tau(v)=0$, then we consider $v\in V_0$ otherwise $v\in V_1$~\cite{conjecturejacko}. Observe that $x_c\in V_0$. For two vertices $v_1(i_1,j_1),v_2(i_2,j_2)\in T_H$, $d(v_1,v_2)$ can be defined as follows.

\begin{eqnarray}
\label{dist}
 d(v_1,v_2)= && \begin{cases} \vert i_1-i_2\vert + \vert j_1-j_2 \vert  & \;if \; \vert i_1-i_2 \vert \leq  \vert j_1-j_2 \vert, \cr 2\vert i_1-i_2 \vert +\tau(v_1) - \tau(v_2), & \;if\; \vert i_i-i_2 \vert > \vert j_1-j_2 \vert .  \cr \end{cases}
 \end{eqnarray}
 The proof of equation~\eqref{dist} is similar as Lemma $3.2$ in~\cite{conjecturejacko}.

Note that for any right vertex $x_r$ and left vertex $x_l$, $D^{2p}_{x_r}$ and $D^{2p}_{x_l}$ are isomorphic. So any property that holds for $D^{2p}_{x_r}$ also holds for $D^{2p}_{x_l}$. Therefore, we will state and prove our results for $D^{2p}_{x_r}$ and these also hold for $D^{2p}_{x_l}$. In following discussion by writing $x$ we actually mean $x_r$.

Let $x(i,j)$ be a right vertex in $V(T_H)$. Note that there are $3k$ vertices which are at distance $k$ from $x$, where $k \in \mathbb{Z^+}$~\cite{conjecturejacko}. Let $\mathcal{F}_{x,k}=\{u \in V(T_H): d(x,u)=k\}$ be the set of those $3k$ vertices. 

The $3k$ vertices $v_{x,k}^1, v_{x,k}^2, \cdots, v_{x,k}^{3k}$ can be partitioned into $6$ disjoint sets $G_{x,k}^1$, $G_{x,k}^2$, $\cdots$, $G_{x,k}^6$. The vertices belong to the sets $G_{x,k}^1, G_{x,k}^2, \cdots, G_{x,k}^6$ with their coordinates are mentioned below. \\
$G_{x,k}^1=\{v_{x,k}^n(i+m-1,j+k-m+1):1\leq n \leq \lceil \dfrac{k}{2} \rceil, 1\leq m \leq \lceil \dfrac{k}{2} \rceil\}$\\
$G_{x,k}^2=\{v_{x,k}^n(i+\lceil \dfrac{k}{2} \rceil,j+\lfloor \dfrac{k}{2} \rfloor-2m + 2): \lceil \frac{k}{2} \rceil+1\leq n \leq k, 1\leq m \leq \lfloor \dfrac{k}{2} \rfloor \},$\\
$G_{x,k}^3=\{v_{x,k}^n(i+\lceil \dfrac{k}{2} \rceil-m+1,j-\lfloor \dfrac{k}{2} \rfloor-m+1):k+1\leq n \leq k+\lceil \frac{k}{2} \rceil, 1\leq m \leq \lceil \dfrac{k}{2} \rceil \},$\\
$G_{x,k}^4=\{v_{x,k}^n(i-m+1,j-k+m-1):k+\lceil \frac{k}{2} \rceil+1\leq n \leq 2k, 1\leq m \leq \lfloor \dfrac{k}{2} \rfloor\},$\\
$G_{x,k}^5=\{v_{x,k}^n(i - \lfloor \dfrac{k}{2} \rfloor,j- \lceil \dfrac{k}{2} \rceil+2m-2):2k+1\leq n \leq 2k+\lceil \frac{k}{2} \rceil,1\leq m \leq \lceil \dfrac{k}{2} \rceil\},$\\
$G_{x,k}^6=\{v_{x,k}^n(i - \lfloor \dfrac{k}{2} \rfloor+m-1,j+ \lceil \dfrac{k}{2} \rceil+m-1):2k+\lceil \frac{k}{2} \rceil+1\leq n \leq 3k,1\leq m \leq \lfloor \dfrac{k}{2} \rfloor\}.$

For illustration, Figure \ref{coordinatesp7} shows the $3k$ vertices for $k=7$ and the corresponding $G_{x,k}^r$, $r=1,2,\cdots, 6$.

\begin{definition}\label{cornerver}
A vertex $v(p,q)\in \mathcal{F}_{x,k}$, $k\geq 2, k\in \mathbb{Z^+}$ is said to be a corner vertex with respect to $x(i,j)$ if one of the following six conditions holds.\\ \\
$1.\ p=i,q=j+k$.\\
$2.\ p=i+\lceil \dfrac{k}{2} \rceil, q=j+\lfloor \dfrac{k}{2} \rfloor$ .\\
$3.\ p=i+\lceil \dfrac{k}{2} \rceil,q=j-\lfloor \dfrac{k}{2} \rfloor$.\\
$4.\ p=i,q=j-k$.\\
$5.\ p=i - \lfloor \dfrac{k}{2} \rfloor,q=j- \lceil \dfrac{k}{2} \rceil$.\\
$6.\ p=i - \lfloor \dfrac{k}{2} \rfloor,q=j+ \lceil \dfrac{k}{2} \rceil$ .\\

\end{definition}

Note that there are six corner vertices in $\mathcal{F}_{x,k}$, $k\geq 2$ and each $G_{x,k}^i$, $1 \leq i \leq 6$, contains one corner vertex. More specifically, $v_{x,k}^1$, $v_{x,k}^{\lceil \frac{k}{2} \rceil+1}$, $v_{x,k}^{k+1}$, $v_{x,k}^{k+\lceil \frac{k}{2} \rceil+1}$, $v_{x,k}^{2k+1}$ and $v_{x,k}^{2k+\lceil \frac{k}{2} \rceil+1}$ are the six corner vertices in $\mathcal{F}_{x,k}$. From Definition~\ref{cornerver}, we get the coordinates of those six corner vertices.
For an example, if we consider $\mathcal{F}_{x,7}$ where coordinates of $x$ is $(0,0)$, the coordinates of the corresponding six corner vertices $v_{x,k}^1$, $v_{x,k}^{5}$, $v_{x,k}^{8}$, $v_{x,k}^{12}$, $v_{x,k}^{15}$ and $v_{x,k}^{19}$ are $(0,7)$, $(4,3)$, $(4,-3)$, $(0,-7)$, $(-3,-4)$ and $(-3,4)$ respectively.
For simplicity, we denote the six corner vertices $v_{x,k}^1$, $v_{x,k}^{\lceil \frac{k}{2} \rceil+1}$, $v_{x,k}^{k+1}$, $v_{x,k}^{k+\lceil \frac{k}{2} \rceil+1}$, $v_{x,k}^{2k+1}$ and $v_{x,k}^{2k+\lceil \frac{k}{2} \rceil+1}$ of $\mathcal{F}_{x,k}$ as $v_{x,k}^{c_1}$  $v_{x,k}^{c_2}$, $v_{x,k}^{c_3}$, $v_{x,k}^{c_4}$, $v_{x,k}^{c_5}$ and $v_{x,k}^{c_6}$ respectively. Also, we denote the set of six corner vertices of $\mathcal{F}_{x,k}$ as $\mathcal{F}_{x,k}^c$.

\begin{definition}\label{noncorner}
The set of vertices which are not corner vertices in $\mathcal{F}_{x,k}$, $k\geq 2, k\in \mathbb{Z^+}$ are said to be non corner vertices.
\end{definition}

The set of non corner vertices of $\mathcal{F}_{x,k}$ is denoted as $\mathcal{F}_{x,k}^{nc}$. Now we have the following definitions.

\begin{definition}
We define $\mathcal{S}_{x,k}^{2h}$ as the set of non corner vertices  in $\mathcal{F}_{x,k}^{nc}$ which are at distance $k$ from $x$ and at distance $2h$ from a corner vertex in $\mathcal{F}_{x,k}^{c}$. That is, $\mathcal{S}_{x,k}^{2h}=\{v: v\in \mathcal{F}_{x,k}^{nc}, \exists u\in \mathcal{F}_{x,k}^{c}, d(v,u)=2h\}$, where $k \geq 5$ and $1 \leq h \leq \lfloor \frac{k}{2} \rfloor -1$.
\end{definition}
Note that when $h=\lfloor \frac{k}{2} \rfloor$, $\mathcal{S}_{x,k}^{2h}$ includes corner vertices and hence no longer a set of non corner vertices.


As an example, for $k=7$, from Figure \ref{coordinatesp7}, we get that $\mathcal{S}_{x,k}^{2}$ $=$ $\{$ $v_{x,k}^{3k}$, $v_{x,k}^{2}$,  $v_{x,k}^{\lceil \frac{k}{2} \rceil}$, $v_{x,k}^{\lceil \frac{k}{2} \rceil+2}$, $v_{x,k}^{k}$, $v_{x,k}^{k+2}$, $v_{x,k}^{k+\lceil \frac{k}{2} \rceil}$, $v_{x,k}^{k+\lceil \frac{k}{2} \rceil+2}$, $v_{x,k}^{2k}$, $v_{x,k}^{2k+2}$, $v_{x,k}^{2k+\lceil \frac{k}{2} \rceil}$, $v_{x,k}^{2k+\lceil \frac{k}{2} \rceil+2}$ $\}$. Note that $h=\lfloor \frac{k}{2} \rfloor=3$ when $k=7$ and $\mathcal{S}_{x,7}^{6}$ includes $v_{x,7}^{c_6}$ which is a corner vertex.


\begin{definition}
We define $\mathcal{U}_{x,k}^{2h}$ as $\displaystyle \mathcal{F}_{x,k}^{c} \bigcup_{r=1}^{h}  \mathcal{S}_{x,k}^{2r}$.
\end{definition}

\begin{definition}\label{rvertex}
For a vertex $v\in V(D_x^{2p})$ and  a set $S \subseteq V(T_H)\setminus V(D_x^{2p})$, $R_v^S=\{u: d(u,v) \geq 2p+1, u \in S\}$ denotes the subset of vertices of $S$ where $f(v)$ can be reused in $S$.
\end{definition}



\begin{figure}[!ht]
\centering
\includegraphics[scale=.5]{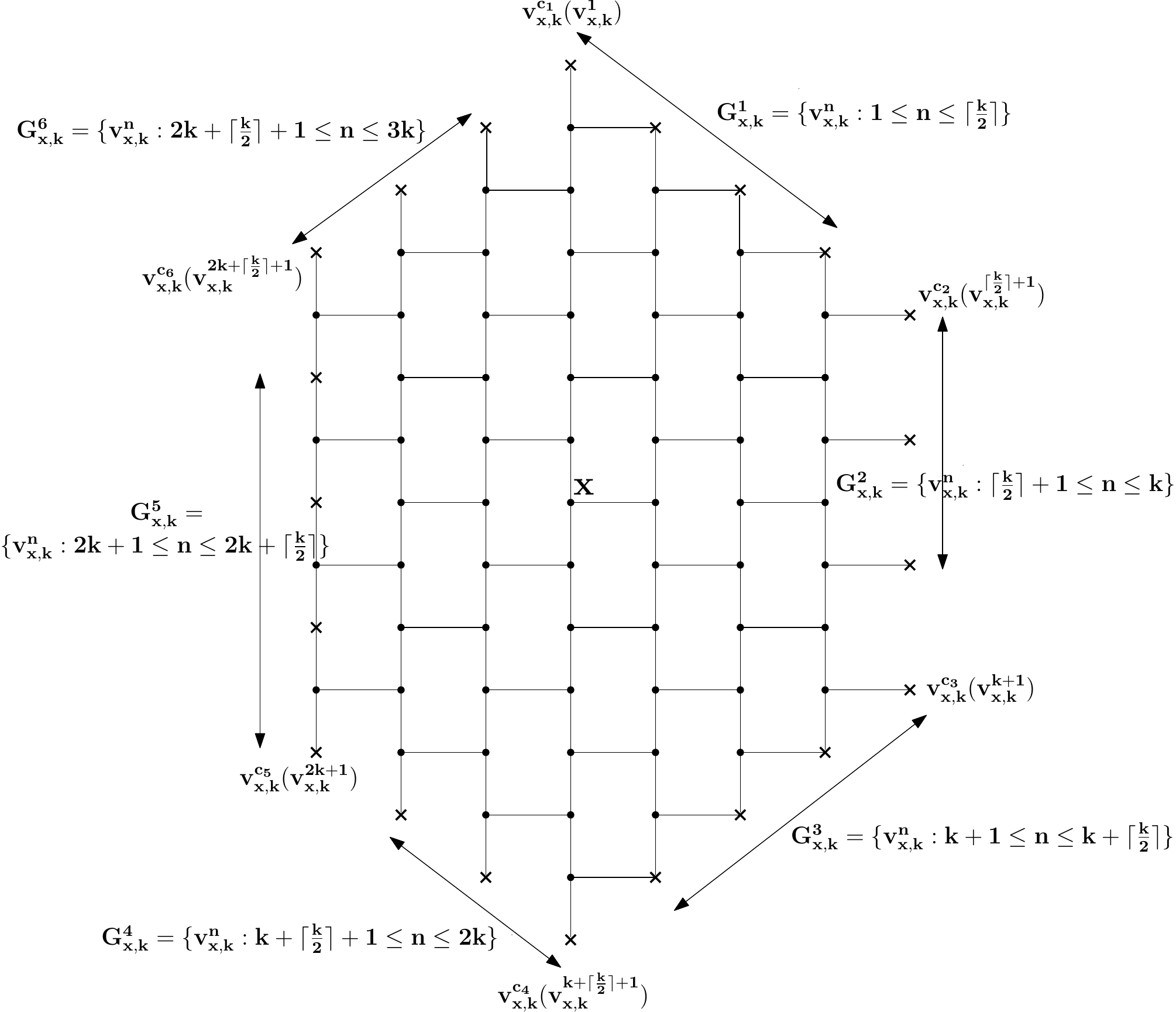}
\caption{The vertices (Marked as $\times $) at distance $k=7$ from $x$.}
\label{coordinatesp7}
\end{figure}


\section{Results}\label{sec:2}

\begin{observation}\label{obs1_new}
For any two distinct vertices $v_1,\;v_2 \in V(D_x^{2p})$, $f(v_1)\neq f(v_2)$ for $2p$ distance coloring.
\end{observation}

\begin{proof}
From Definition \ref{defold5}, in $D_x^{2p}$, $d(v_1,v_2)\leq 2p$ for any two distinct vertices $v_1,\;v_2 \in V(D_x^{2p}$. Hence the proof.
\end{proof}

\begin{observation}\label{obs2_new}
The number of vertices in $D_x^{2p}$ is $|D_x^{2p}|=1+\frac{3p(p+1)}{2}$. 
\end{observation}

\begin{proof}
There are $3k$ vertices which are at distance $k$ from $x$ in $D_x^{2p}$, where $k=1,2,\cdots, p$. So total number of vertices in $D_x^{2p}$ is $1+3\times 1+ 3 \times 2+ \cdots + 3\times p =1+\frac{3p(p+1)}{2}$. Hence the proof.
\end{proof}

\begin{observation}\label{obs3_new}
$1+\frac{3p(p+1)}{2}+\lfloor \frac{p}{2} \rfloor=\left[  \dfrac{3}{8} \left( \, 2p+\dfrac{4}{3} \right) ^2 \right]$, where $x\in \mathbb{R}$,  $[x]$ is an integer satisfying $x-\dfrac{1}{2} < [x] \leq x+\dfrac{1}{2}$ and $\lfloor x \rfloor$ is the largest integer less than or equals to $x$.
\end{observation}

\begin{proof}
When $p=2q$: Left hand side is $1+\frac{6q(2q+1)}{2} + \lfloor \frac{2q}{2} \rfloor = 6q^2+4q+1$ and right hand side is $\left[  \dfrac{3}{8} \left( \, 4q+\dfrac{4}{3} \right)^2 \right]=\left[ 6q^2+4q+\frac{2}{3} \right]= 6q^2+4q+1$ (since $\frac{2}{3} > \frac{1}{2}$).

When $p=2q+1$: Left hand side is $1+\frac{3(2q+1)(2q+2)}{2} + \lfloor \frac{2q+1}{2} \rfloor = 6q^2+10q+4$ and right hand side is $\left[  \dfrac{3}{8} \left( \, 4q+ 2 + \dfrac{4}{3} \right)^2 \right]=\left[ 6q^2+10q+4+\frac{1}{6} \right]=6q^2+10q+4$ (since $\frac{1}{6} < \frac{1}{2}$). Hence the proof.
\end{proof}

\begin{observation}\label{obs1}
For two vertices $v_1\in V(D_x^{2p})$ and $v_2\in V(T_H) \setminus V(D_x^{2p})$ where $d(v_1,x)=d_1$ and $d(v_2,x)=d_2$, if $d_1+d_2<2p+1$ then $f(v_1)\neq f(v_2)$ for $2p$ distance coloring.
\end{observation}

\begin{proof}
There exists a path between $v_1$ and $v_2$ through $x$ and length of the path is $(d_1+d_2)<2p+1$. So, $d(v_1,v_2)< 2p+1$. Hence the proof.
\end{proof}

\begin{observation}\label{cornerrepeatp}
For  any corner vertex $v \in \mathcal{F}_{x,p-q}^{c}$, $f(v)$ can be reused at most twice in $\mathcal{F}_{x,p+q+1}$, where $q$ is a non-negative integer and $q=0,1, \cdots, p-2$.
\end{observation}

\begin{proof}
 Consider the vertex $v=v_{x,p-q}^{c_1}$ in $\mathcal{F}_{x,p-q}^{c}$. Let $k=p+q+1$. Using equation~\eqref{dist}, it can be shown that the set of vertices in $\mathcal{F}_{x,k}$ where $v$ can be reused is given by $R_v^{\mathcal{F}_{x,k}}=G_{x,k}^3 \cup G_{x,k}^4 \cup v_{x,k}^{c_5}$. Note that $R_v^{\mathcal{F}_{x,k}}=X\cup Y$ where $X=G_{x,k}^3$ and $Y=G_{x,k}^4 \cup v_{x,k}^{c_5}$. Observe that $\nexists u_1,\;u_2\in X$ such that $d(u_1,u_2)\geq 2p+1$ and $\nexists w_1,\;w_2\in Y$ such that  $d(w_1,w_2)\geq 2p+1$. Hence $f(v)$ can be reused once in $X$ and once in $Y$. Similarly we can prove that for any other corner vertex $w \in \mathcal{F}_{x,p-q}^c$, $f(w)$ can be reused at most twice in $\mathcal{F}_{x,k}$. Hence the proof. 
\end{proof}

\begin{observation}\label{noncornerrepeat1}
For any non corner vertex $v \in \mathcal{F}_{x,p-q}^{nc}$, $f(v)$ can be reused at most once in $\mathcal{F}_{x,p+q+1}$, where $q$ is a non-negative integer and $q=0,1, \cdots, p-3$.
\end{observation}

\begin{proof}
Let us consider any non corner vertex $v \in G_{x,p-q}^6$. Let $k=p+q+1$. Using equation~\eqref{dist}, we can shown that the set of vertices in $\mathcal{F}_{x,k}$ where $v$ can be reused is given by $R_v^{\mathcal{F}_{x,k}}=G_{x,k}^3 \cup \{v_{x,k}^{c_4}\}$. It can now be observed that $ \nexists \; u_1,u_2 \in R_v^{\mathcal{F}_{x,k}}$ such that $d(u_1,u_2)\geq 2p+1$. That is, $v$ can be reused only once in $\mathcal{F}_{x,k}$. For other non corner vertices in $G_{x,p-q}^1, G_{x,p-q}^2, G_{x,p-q}^3, G_{x,p-q}^4, G_{x,p-q}^5$ the same result can be proved considering $G_{x,k}^4 \cup \{v_{x,k}^{c_5}\}, G_{x,k}^5 \cup \{v_{x,k}^{c_6}\}, G_{x,k}^6 \cup \{v_{x,k}^{c_1}\}, G_{x,k}^1 \cup \{v_{x,k}^{c_2}\}, G_{x,k}^2 \cup \{v_{x,k}^{c_3}\}$ respectively. Hence the proof.  
\end{proof}

\begin{theorem}\label{lem1}
At least a new color which is not used in $D_x^{2p}$ must be introduced to color the vertices of $\mathcal{F}_{x,p+1}$ for $2p$ distance coloring.
\end{theorem}
\begin{proof}
From Observation~\ref{obs1},  $f(v_1) \neq f(v_2)$ where $v_1 \in \mathcal{F}_{x,p+1}$, $v_2 \in \mathcal{F}_{x,q}$ and $q=0,1, \cdots p-1$ for $2p$ distance coloring. So, colors of vertices of $\mathcal{F}_{x,p}$ can  only be reused at the vertices of $\mathcal{F}_{x,p+1}$. Note that $\vert V(\mathcal{F}_{x,p}) \vert=3p$ and $\vert V( \mathcal{F}_{x,p+1}) \vert=3(p+1)$. 

So colors of some vertices of $\mathcal{F}_{x,p}$ must be reused more than once in  $\mathcal{F}_{x,p+1}$. 

From Observation \ref{cornerrepeatp}, we get that only the colors of the six corner vertices $v_{x,p}^{c_1}$, $v_{x,p}^{c_2}$, $\cdots$, $v_{x,p}^{c_6}$ of $\mathcal{F}_{x,p}$ can be reused twice in  $\mathcal{F}_{x,p+1}$. The only possibility to reuse $f(v_{x,p}^{c_1})$ twice in $\mathcal{F}_{x,p+1}$ is to use $f(v_{x,p}^{c_1})$ in $v_{x,p+1}^{c_3}$ and $v_{x,p+1}^{c_5}$.  That is, $f(v_{x,p+1}^{c_3})=f(v_{x,p}^{c_1})$ and $f(v_{x,p+1}^{c_5})=f(v_{x,p}^{c_1})$.  Similarly, if $f(v_{x,p}^{c_2}), f(v_{x,p}^{c_3}), \cdots, f(v_{x,p}^{c_6})$ are to be reused twice in $\mathcal{F}_{x,p+1}$, then
$f(v_{x,p+1}^{c_4})=f(v_{x,p+1}^{c_6})=f(v_{x,p}^{c_2})$; $f(v_{x,p+1}^{c_5})=f(v_{x,p+1}^{c_1})=f(v_{x,p}^{c_3})$; $f(v_{x,p+1}^{c_6})=f(v_{x,p+1}^{c_2})=f(v_{x,p}^{c_4})$; $f(v_{x,p+1}^{c_1})=f(v_{x,p+1}^{c_3})=f(v_{x,p}^{c_4})$; and $f(v_{x,p+1}^{c_2})=f(v_{x,p+1}^{c_4})=f(v_{x,p}^{c_6})$.

From Observation \ref{noncornerrepeat1}, we get that the color of any non corner vertex of $\mathcal{F}_{x,p}$ can be reused at most once in $\mathcal{F}_{x,p+1}$. So if the vertices of $\mathcal{F}_{x,p+1}$ are to be colored only with the colors used $\mathcal{F}_{x,p}$, the colors of at least three corner vertices of $\mathcal{F}_{x,p}$ must be reused twice each in  $\mathcal{F}_{x,p+1}$.

Note that any two of $f(v_{x,p}^{c_1})$, $f(v_{x,p}^{c_3})$ and $f(v_{x,p}^{c_5})$  can not be reused twice each simultaneously in $\mathcal{F}_{x,p+1}$ as in that case they must be reused in a common vertex in $\mathcal{F}_{x,p+1}$ which is not possible.  Similarly, any two of $f(v_{x,p}^{c_2})$, $f(v_{x,p}^{c_4})$ and $f(v_{x,p}^{c_6})$  can not be reused twice each simultaneously in $\mathcal{F}_{x,p+1}$ as in that case too they must be reused in a common vertex in $\mathcal{F}_{x,p+1}$ which is also not possible. So the colors of at most two corner vertices of $\mathcal{F}_{x,p}$, one from $\{v_{x,p}^{c_1}$, $v_{x,p}^{c_3}$, $v_{x,p}^{c_5}\}$ and one from $\{v_{x,p}^{c_2}$, $v_{x,p}^{c_4}$, $v_{x,p}^{c_6}\}$ can be reused twice each in $\mathcal{F}_{x,p+1}$.
In other words, at least one vertex remains uncolored in $\mathcal{F}_{x,p+1}$. Hence at least a new color which is not used in $V(D_x^{2p})$ must be introduced in $\mathcal{F}_{x,p+1}$.
\end{proof}



\begin{observation}\label{noncornerrepeatp_1} For any $v \in \mathcal{S}_{x,p-q}^{2r}$, $f(v)$  can be reused at most twice in $\displaystyle \bigcup_{h=1}^{2r+1} \mathcal{F}_{x,p+q+h}$, where $q=0,1,\cdots,p-4$ and $r=1, 2, \cdots, \lfloor \frac{p-q}{2} \rfloor -1$.\\ 

\end{observation}

\begin{proof}
Let us consider the vertex $v=v_{x,p}^{3p}(i-1,j+p-1) \in \mathcal{S}_{x,p}^{2}$, where $r=1$ and $q=0$. Using equation~\eqref{dist}, it can be shown that in $\mathcal{F}_{x,p+1}$, $f(v)$ can only be reused in $G_{x,p+1}^3 \cup \{v_{x,p+1}^{c_4}\}$. That is, $R_v^{\mathcal{F}_{x,p+1}}=G_{x,p+1}^3 \cup \{v_{x,p+1}^{c_4}\}$. 
Similarly it can be shown that in $\mathcal{F}_{x,p+2}$, $f(v)$ can only be reused in $R_v^{\mathcal{F}_{x,p+2}}=G_{x,p+2}^3 \cup \{v_{x,p+2}^{c_4}\}$. 
Similarly it can be shown that in $\mathcal{F}_{x,p+3}$, $f(v)$ can only be reused in $R_v^{\mathcal{F}_{x,p+3}}=\{v_{x,p+3}^{p
+3}\} \cup G_{x,p+3}^3 \cup G_{x,p+3}^4 \cup \{v_{x,p+3}^{c_5}\}$. Now $R_v^{\mathcal{F}_{x,p+1}} \cup R_v^{\mathcal{F}_{x,p+2}} \cup R_v^{\mathcal{F}_{x,p+3}}$
can be partitioned into two disjoint sets $X=R_v^{\mathcal{F}_{x,p+1}} \cup R_v^{\mathcal{F}_{x,p+2}} \cup \{v_{x,p+3}^{p
+3}\} \cup G_{x,p+3}^3$ and $Y=G_{x,p+3}^4 \cup \{v_{x,p+3}^{c_5}\}$, where $\nexists u_1, u_2 \in X$ such that $u_1,u_2\geq 2p+1$ and $\nexists w_1, w_2 \in Y$ such that $w_1,w_2\geq 2p+1$. So $f(v)$ can be reused at most twice in $R_v^{\mathcal{F}_{x,p+1}} \cup R_v^{\mathcal{F}_{x,p+2}} \cup R_v^{\mathcal{F}_{x,p+3}}$, once in $X$ and once in $Y$. With similar argument, we can prove the same for any other vertex $v\in \mathcal{S}_{x,p-q}^{2r}$. 
\end{proof}

\begin{observation}\label{noncornerrepeatp_2}
For any $v \in \displaystyle \mathcal{F}_{x,p-q}^{nc} \setminus \mathcal{S}_{x,p-q}^{2r}$, $f(v)$  can be reused at most once in $\displaystyle \bigcup_{h=1}^{2r+1} \mathcal{F}_{x,p+q+h}$, where $q=0,1, \cdots, p-5$ and $r=1, 2, \cdots, \lfloor \frac{p-q}{2} \rfloor -1$.
\end{observation}

\begin{proof}

Consider the re-usability of the colors of the vertices of $ \mathcal{F}_{x,p}^{nc} \setminus \mathcal{S}_{x,p}^2$ in $ \displaystyle \bigcup_{h=1}^3 \mathcal{F}_{x,p+h}$, where $r=1$ and $q=0$. Consider the vertex $u=v_{x,p}^3(i+2,j+p-2) \in \mathcal{F}_{x,p}^{nc}\setminus \mathcal{S}_{x,p}^2 $. In $ \displaystyle \bigcup_{h=1}^3 \mathcal{F}_{x,p+h}$, $f(u)$ can only be reused in $G_{x,p+1}^4 \cup \{v_{x,p+1}^{c_5}\} \cup G_{x,p+2}^4 \cup \{v_{x,p+2}^{c_5}\} \cup  \{v_{x,p+3}^{p+3+ \lceil \frac{p+3}{2} \rceil} \}\cup G_{x,p+3}^4 \cup \{v_{x,p+3}^{c_5}\} \cup \{v_{x,p+3}^{2(p+3)+2}\}  $. It may be observed that $\nexists u_1,u_2 \in R_v^{\mathcal{F}_{x,p+1} \cup \mathcal{F}_{x,p+2}\cup \mathcal{F}_{x,p+3}}$ such that $d(u_1,u_2)>2p+1$. Therefore $f(u)$ can be reused at most once in $ \displaystyle \bigcup_{h=1}^3 \mathcal{F}_{x,p+h}$. For any other vertex in $\mathcal{F}_{x,p-q}^{nc}\setminus \mathcal{S}_{x,p-q}^{2r}$ the same result can be proved similarly. Hence the proof.
\end{proof}

\begin{theorem}\label{2ndnew}
A second new color which is not used in $V(D_x^{2p}) \cup \mathcal{F}_{x,p+1}$ must be introduced to color the vertices of $\mathcal{F}_{x,p+2}\cup \mathcal{F}_{x,p+3}$ for $2p$ distance coloring.
\end{theorem}

\begin{proof}
We consider the re-usability of the colors of $\displaystyle \bigcup_{r=0}^p \mathcal{F}_{x,r}$ in $\displaystyle \bigcup_{q=1}^3 \mathcal{F}_{x,p+q}$. From Observation~\ref{obs1}, we get that the colors used in $\displaystyle \bigcup_{r=0}^{p-3} \mathcal{F}_{x,r}$ can not be reused in $\displaystyle \bigcup_{q=1}^3 \mathcal{F}_{x,p+q}$. Again from Observation~\ref{obs1}, the colors used in $\mathcal{F}_{x,p}$, $\mathcal{F}_{x,p-1}$ and $\mathcal{F}_{x,p-2}$   can be reused in $\displaystyle \bigcup_{q=1}^3 \mathcal{F}_{x,p+q}$, $\displaystyle \bigcup_{q=2}^3 \mathcal{F}_{x,p+q}$ and $ \mathcal{F}_{x,p+3}$ respectively. 

From the proof of Theorem~\ref{lem1}, at least a vertex in $\mathcal{F}_{x,p+1}$ (say $u$) can not be colored by the colors used in $V(D_x^{2p})$. In other words, a new color which is not used in $V(D_x^{2p})$ must be introduced in $\mathcal{F}_{x,p+1}$.  From Observation~\ref{noncornerrepeat1}, the color of a non corner vertex in $\mathcal{F}_{x,p}$ can be reused at most once in $\mathcal{F}_{x,p+1}$. If the color of a non corner vertex in $\mathcal{F}_{x,p}$ is not reused in $\mathcal{F}_{x,p+1}$, then there must exists another vertex other than $u$ which can not be colored with the colors used in $V(D_x^{2p})$. 

From Observation~\ref{cornerrepeatp}, the color of a corner vertex in $\mathcal{F}_{x,p}$ can be reused at most twice in $\mathcal{F}_{x,p+1}$. If the color of a corner vertex in $\mathcal{F}_{x,p}$ is reused only once or not reused in $\mathcal{F}_{x,p+1}$, then there must exists one more vertex or two more vertices respectively, other than $u$, which can not be colored with the colors used in $V(D_x^{2p})$. Again from Theorem \ref{lem1} we get that at most $2$ corner vertices of $\mathcal{F}_{x,p}$ can be reused twice each in $\mathcal{F}_{x,p+1}$. Thus the remaining $4$ corner vertices can be reused at most once in $\mathcal{F}_{x,p+1}$. Now the new color can be used at most $3$ times in 
$\displaystyle \bigcup_{q=1}^3 \mathcal{F}_{x,p+q}$. If the new color is used $1$, $2$ or $3$ times in $\mathcal{F}_{x,p+1}$ then it can be reused $2$, $1$ or $0$ times respectively in $\displaystyle \bigcup_{q=2}^3 \mathcal{F}_{x,p+q}$. Since our proof is based on the number of vertices of $\displaystyle \bigcup_{q=1}^3 \mathcal{F}_{x,p+q}$ which remain uncolored by the colors used in $V(D_x^{2p})$ and the new color used in $\mathcal{F}_{x,p+1}$, we can assume that all non corner vertices of $\mathcal{F}_{x,p}$ is reused once in $\mathcal{F}_{x,p+1}$,  $2$ corner vertices of $\mathcal{F}_{x,p}$ are reused twice each in $\mathcal{F}_{x,p+1}$ and remaining $4$ corner vertices  of $\mathcal{F}_{x,p}$ are used only once in $\mathcal{F}_{x,p+1}$. If any of them is not reused in $\mathcal{F}_{x,p+1}$, it may color one more vertex in $\displaystyle \bigcup_{q=2}^3 \mathcal{F}_{x,p+q}$, but at the same time it will create one uncolored vertex in $\mathcal{F}_{x,p+1}$. As we are counting uncolored vertices in $\displaystyle \bigcup_{q=1}^3 \mathcal{F}_{x,p+q}$, there is no benefit of not reusing any of them in $\mathcal{F}_{x,p+1}$.



From the proof of Observation~\ref{noncornerrepeatp_1}, the color of each vertex of $\mathcal{S}_{x,p}^2$ may be reused once more in $\mathcal{F}_{x,p+3}$ even after using once in $\mathcal{F}_{x,p+1}$. Again from Observation \ref{noncornerrepeatp_2} we get that $\displaystyle \bigcup_{q=0}^2 \mathcal{F}_{x,p-q}^{nc} \setminus \mathcal{S}_{x,p}^2$ can not reused any more in $\displaystyle \bigcup_{r=1}^{3} \mathcal{F}_{x,p+r}$. Let $S$ be the subset of $\mathcal{F}_{x,p}^c$ such that for every $v \in S$, $f(v)$ has not been reused twice in $\mathcal{F}_{x,p+1}$. Since as per our assumption, $2$ corner vertices have already been reused twice each in $\mathcal{F}_{x,p+1}$, there are $4$ vertices in $S$, each of which may be reused once more in $\mathcal{F}_{x,p+2} \cup \mathcal{F}_{x,p+3}$. Note that there are $3(p+2)+3(p+3)=6p+15$ vertices in $\mathcal{F}_{x,p+2} \cup \mathcal{F}_{x,p+3}$ which can potentially be colored by the colors used in $\mathcal{F}_{x,p-1} \cup \mathcal{F}_{x,p-2} \cup \mathcal{S}_{x,p}^2 \cup S$.

Observe that there are $(3(p-1)-6)+(3(p-2)-6)=6p-21$ non corner vertices in $\mathcal{F}_{x,p-1} \cup \mathcal{F}_{x,p-2}$ whose colors can be reused at most once each in $\mathcal{F}_{x,p+2} \cup \mathcal{F}_{x,p+3}$. So, using the colors of them together, we can color at most $6p-21$ vertices of $\mathcal{F}_{x,p+2} \cup \mathcal{F}_{x,p+3}$. If $f(v)$, $v \in \mathcal{S}_{x,p}^2$, is to be reused once in $u_1 \in \mathcal{F}_{x,p+1}$ and once in $u_2\in \mathcal{F}_{x,p+3}$ then $u_1 \in \mathcal{F}_{x,p+1}^c \cup \mathcal{S}_{x,p+1}^{2}$ and $u_2 \in \mathcal{F}_{x,p+3}^c \cup \mathcal{S}_{x,p+3}^{2}$. Moreover, if $f(v)$ is reused at $u_1 \in \mathcal{S}_{x,p+1}^{2}$ then it must be reused at $u_2 \in \mathcal{F}_{x,p+3}^c$. As there are $6$ corner vertices in $\mathcal{F}_{x,p+3}$, the colors of $6$ out of the $12$ vertices of $\mathcal{S}_{x,p}^{2}$ can be reused in $\mathcal{S}_{x,p+1}^{2}$ and the remaining $6$ must be reused in  $\mathcal{S}_{x,p+3}^{2}$. Note that, if the colors of all the corner vertices of $\mathcal{F}_{x,p-2}^c$ are to be reused twice each in $\mathcal{F}_{x,p+3}$ and colors of all the vertices of $\mathcal{S}_{x,p}^{2}$ are to be reused once each in $\mathcal{F}_{x,p+3}$ then at least $12$ of them must be reused in corner vertices of $\mathcal{F}_{x,p+3}$. But there are $6$ corner vertices in $\mathcal{F}_{x,p+3}$. So, using all of them together, we can color at most $(6\times 2 + 12\times 1) - 6 = 18$ vertices in $\mathcal{F}_{x,p+3}^c \cup \mathcal{S}_{x,p+3}^{2}$.

Again from Observation \ref{cornerrepeatp}, the colors of $6$ corner vertices of $\mathcal{F}_{x,p-1}^c$ can be reused at most twice each in $\mathcal{F}_{x,p+2} \cup \mathcal{F}_{x,p+3}$. So their colors together can be reused at most $12$ vertices in $\mathcal{F}_{x,p+2} \cup \mathcal{F}_{x,p+3}$. Note that the $6$ corner vertices of $\mathcal{F}_{x,p+2}^c$ must be colored by them if the colors of $\mathcal{F}_{x,p-2}^c \cup \mathcal{S}_{x,p}^2$ are reused in $\mathcal{F}_{x,p+3}$ with their maximum re-usability. Therefore all together using the colors used in $\mathcal{F}_{x,p-1}\cup \mathcal{F}_{x,p-2}\cup \mathcal{S}_{x,p}^2$, we can color at most $(6p-21)+18+12=6p+9$ vertices in $\mathcal{F}_{x,p+2} \cup \mathcal{F}_{x,p+3}$. 
Now there are $(6p+15)-(6p+9)=6$ vertices of $\mathcal{F}_{x,p+2} \cup \mathcal{F}_{x,p+3}$ which are yet to be colored. These $6$ vertices may be colored by the colors used in the $4$ vertices of $S$. 

Since the colors of two corner vertices of $\mathcal{F}_{x,p}^c$ have already been reused in the $4$ corner vertices of  $\mathcal{F}_{x,p+1}$,  there are only  $2$ remaining corner vertices in $\mathcal{F}_{x,p+1}$ where colors of $S$ may be reused. If $f(v)$ ($v\in S$) is reused in a corner vertex of $\mathcal{F}_{x,p+1}^c$ then $f(v)$ may be reused once again in $\mathcal{F}_{x,p+3}^c \cup \mathcal{S}_{x,p+3}^2$ or $ \mathcal{F}_{x,p+2}^c \cup \mathcal{S}_{x,p+2}^2$. But the vertices of $\mathcal{F}_{x,p+3}^c \cup \mathcal{S}_{x,p+3}^2$ have already been colored. So here we consider the re usability of $f(v)$ in $ \mathcal{F}_{x,p+2}^c \cup \mathcal{S}_{x,p+2}^2$. If $f(v)$ is reused in a non corner vertex in $\mathcal{F}_{x,p+1}^c$ then it may only be reused in a corner vertex of $\mathcal{F}_{x,p+2}$.  But these vertices have already been colored by the colors of the corner vertices of $\mathcal{F}_{x,p-1}^c$. So the only possibility remaining is to reuse $f(v)$ in a corner vertex of $\mathcal{F}_{x,p+1}^c$. Since there are only $2$ corner vertices remaining in $\mathcal{F}_{x,p+1}$, at most $2$ of the $4$ vertices of $S$ can be reused once each in $\mathcal{F}_{x,p+2}$. Therefore at least $6-2=4$ vertices in $\mathcal{F}_{p+2}\cup \mathcal{F}_{p+3}$ can not be colored by the colors used in $D_x^{2p}$. Note that the new color introduced in $\mathcal{F}_{x,p+1}$ can be reused at most twice more in $\mathcal{F}_{x,p+1} \cup \mathcal{F}_{x,p+2} \cup \mathcal{F}_{x,p+3}$. Hence to color the remaining $(4-2)=2$ vertices in $\mathcal{F}_{x,p+2} \cup \mathcal{F}_{x,p+3}$, a second new color must be introduced. Hence the proof.
\end{proof}

\begin{theorem}\label{kthnew}
$\lfloor \frac{p}{2} \rfloor$ new colors must be introduced to color the vertices of $\mathcal{F}_{x,p+2\lfloor \frac{p}{2} \rfloor-1}$ for $2p$ distance coloring.
\end{theorem}

\begin{proof}
From Theorem ~\ref{lem1}, we get that a new color is required to color the vertices of $\mathcal{F}_{x,p+1}$. Again from Theorem~\ref{2ndnew}, a second new color is required to color the vertices of $\mathcal{F}_{x,p+2}\cup \mathcal{F}_{x,p+3}$. We assume that $r$ th new color has been used to color the vertices of $\displaystyle \bigcup_{q=1}^{2r-1}\mathcal{F}_{x,p+q}$ and then we will show that an $r+1$ th new color will be required to color the vertices of $\mathcal{F}_{x,p+2r}\cup \mathcal{F}_{x,p+2r+1}$. 

From the Observations \ref{cornerrepeatp}, \ref{noncornerrepeat1}, \ref{noncornerrepeatp_1} and \ref{noncornerrepeatp_2}, we get that the colors used in the vertices $\mathcal{F}_{x,p-2r}\cup \mathcal{S}_{x,p}^{2r} \cup \mathcal{S}_{x,p-2}^{2(r-1)}\cup \cdots \cup \mathcal{S}_{x,p-2(r-1)}^{2(r-(r-1))}$ can be reused in $\mathcal{F}_{x,p+2r+1}$, where $r\leq \lfloor \frac{p}{2} \rfloor -1$. Let $S_1=\mathcal{S}_{x,p}^{2r} \cup \mathcal{S}_{x,p-2}^{2(r-1)}\cup \cdots \cup \mathcal{S}_{x,p-2(r-1)}^{2(r-(r-1))}$. Observe that there are $3(p-2r)-6$ and $6$ non corner and corner vertices in $\mathcal{F}_{x,p-2r}$ respectively. From Observations \ref{cornerrepeatp} and \ref{noncornerrepeat1}, the color of a corner and a non corner vertex of $\mathcal{F}_{x,p-2r}$ can be reused twice and once  in $\mathcal{F}_{x,p+2r+1}$ respectively. From similar discussion in the proof of Theorem \ref{2ndnew}, we assume that colors of all the vertices of $\mathcal{S}_{x,p-2k}^{2(r-k)}$, $0\leq k \leq r-1$ have been reused in $\mathcal{F}_{x,p+2k+1}$. Now color of a vertex of $\mathcal{S}_{x,p-2k}^{2(r-k)}$, $0\leq k \leq r-1$, if reused in $\mathcal{F}_{x,p+2k+1}$ may be reused at most once more in $\mathcal{F}_{x,p+2r+1}$.  

Since each $\mathcal{S}_{x,p-2k}^{2(r-k)}$, $0\leq k \leq r-1$, contains $12$ vertices (for $p-2k \geq 4$), there are total  $12r$ vertices, all of which may be reused in $\mathcal{F}_{x,p+2r+1}$. From the discussion stated in the proof of Theorem~\ref{2ndnew}, note that, if the colors of all the corner vertices of $\mathcal{F}_{x,p-2r}$ are to be reused twice each in $\mathcal{F}_{x,p+2r+1}$ and colors of all the vertices of $\mathcal{S}_{x,p-2k}^{2(r-k)}$, $0\leq k \leq r-1$, are to be reused once each in $\mathcal{F}_{x,p+2r+1}$ then at least $12$ of them must be reused in corner vertices of $\mathcal{F}_{x,p+2r+1}$. But there are only $6$ corner vertices in $\mathcal{F}_{x,p+2r+1}$. So all of them together may be reused in $12r+6\times 2-6=12r+6$ vertices in $\mathcal{F}_{x,p+2r+1}$ and from the discussion stated in the proof of Theorem~\ref{2ndnew}, they must be reused in $12r+6$ vertices of $\mathcal{U}_{x,p+2r+1}^{2r}$. So the colors of the vertices of $\mathcal{F}_{x,p-2r}\cup S_1$ can be reused in $(3(p-2r)-6)\times 1 + 12r+6=3p+6r$ vertices of $\mathcal{F}_{x,p+2r+1}$ but   $3(p+2r+1)=3p+6r+3$ vertices are there. So, at least $3$ vertices in $\mathcal{F}_{x,p+2r+1}$ should remain uncolored.

From the Observations \ref{cornerrepeatp}, \ref{noncornerrepeat1}, \ref{noncornerrepeatp_1} and \ref{noncornerrepeatp_2}, we get that the colors used in the vertices of $\mathcal{F}_{x,p-2r+1}\cup \mathcal{S}_{x,p-1}^{2(r-1)} \cup \mathcal{S}_{x,p-3}^{2(r-2)}\cup \cdots \cup \mathcal{S}_{x,p-1-2(r-2)}^{2((r-1)-(r-2))}$ can be reused in $\mathcal{F}_{x,p+2r} \cup \mathcal{F}_{x,p+2r+1} $, where $r\leq \lfloor \frac{p}{2} \rfloor -1$. Let $S_2=\mathcal{S}_{x,p-1}^{2(r-1)} \cup \mathcal{S}_{x,p-3}^{2(r-2)}\cup \cdots \cup \mathcal{S}_{x,p-1-2(r-2)}^{2((r-1)-(r-2))}$.
Observe that there are $3(p-2r+1)-6$ and $6$ non corner and corner vertices in $\mathcal{F}_{x,p-2r+1}$ respectively. From Observations \ref{cornerrepeatp} and \ref{noncornerrepeat1}, color of a corner and a non corner vertex of $\mathcal{F}_{x,p-2r+1}$ can be reused twice and once in $\mathcal{F}_{x,p+2r} \cup \mathcal{F}_{x,p+2r+1} $ respectively. Here also we assume that colors of all the vertices of $\mathcal{S}_{x,p-1-2k}^{2(r-1-k)}$, $0\leq k \leq r-2$, have been reused in $\mathcal{F}_{x,p+2k+2}$. Now color of a vertex of $\mathcal{S}_{x,p-1-2k}^{2(r-1-k)}$, $0\leq k \leq r-2$, if reused in $\mathcal{F}_{x,p+2k+2}$ may be reused once more in $\mathcal{F}_{x,p+2r} \cup \mathcal{F}_{x,p+2r+1}$.  Since each $\mathcal{S}_{x,p-1-2k}^{2(r-1-k)}$, $0\leq k \leq r-2$, contains $12$ vertices ($p-1-2k\geq 4$), there are total $12(r-1)$ vertices, all of which may be reused in $\mathcal{F}_{x,p+2r}\cup \mathcal{F}_{x,p+2r+1}$. So the colors of the vertices of $\mathcal{F}_{x,p-2r+1}\cup S_2$ together can be reused in $(3(p-2r+1)-6)\times 1+ 6\times 2 +12(r-1)=3p+6r-3$ vertices of $\mathcal{F}_{x,p+2r+1}\cup \mathcal{F}_{x,p+2r}$.

Therefore using the colors of the vertices of $(\mathcal{F}_{x,p-2r} \cup \mathcal{F}_{x,p-2r+1}) \cup S_1 \cup S_2$ together, we can color $(3p+6r)+(3p+6r-3)=6p+12r-3$ vertices of $\mathcal{F}_{x,p+2r} \cup \mathcal{F}_{x,p+2r+1}$. But $3(p+2r)+3(p+2r+1)=6p+12r+3$ vertices are there. So $(6p+12r+3)-(6p+12r-3)=6$ vertices remain uncolored. Using similar discussion stated in the proof of Theorem~\ref{2ndnew}, there are $6$ colors which can not be reused in $\mathcal{F}_{x,p+2r-2} \cup \mathcal{F}_{x,p+2r-1}$. Hence they may be reused in $\mathcal{F}_{x,p+2r} \cup \mathcal{F}_{x,p+2r+1}$ and in that case, if they are to be reused in $\mathcal{F}_{x,p+2r+1}$, they must be reused in $\mathcal{U}_{x,p+2r+1}^{2r}$ . Note that the vertices of $\mathcal{U}_{x,p+2r+1}^{2r}$ have already been colored by the colors of the vertices of $\mathcal{F}_{x,p-2r}\cup S_1$.  So these $6$ colors can not be reused in $ \mathcal{F}_{x,p+2r+1}$. So they must be reused in $\mathcal{F}_{x,p+2r}$. If they are to be reused in $\mathcal{F}_{x,p+2r}$, then those $6$ colors must be reused in $\mathcal{U}_{x,p+2r}^{2r}$  due to the reusing of colors of  $\mathcal{F}_{x,p-2r+1}\cup \mathcal{F}_{x,p-2r} \cup S_1 \cup S_2$ in $\mathcal{F}_{p+2r}\cup \mathcal{F}_{p+2r+1}$. Again, from similar discussion in the proof of Theorem \ref{2ndnew},, if all the colors used in  $\mathcal{F}_{x,p-2r+1}\cup \mathcal{F}_{x,p-2r} \cup S_1 \cup S_2$ are to be reused in $\mathcal{F}_{p+2r}\cup \mathcal{F}_{p+2r+1}$, then at most $3$ vertices remain uncolored in $\mathcal{U}_{x,p+2r}^{2r}$. So out of the $6$ colors, at most $3$ can be reused in $\mathcal{F}_{x,p+2r}$ and hence at least $(6-3)=3$ vertices in $\mathcal{F}_{p+2r}\cup \mathcal{F}_{p+2r+1}$ remain uncolored. 

Note that the $r$ th new color introduced and used twice in $\mathcal{F}_{x,p+2r-2} \cup \mathcal{F}_{p+2r-1}$ can be reused at most twice more in $\mathcal{F}_{x,p+2r} \cup \mathcal{F}_{p+2r+1}$. Hence to color the remaining $(3-2)=1$ vertices in $\mathcal{F}_{x,p+2r} \cup \mathcal{F}_{p+2r+1}$ the $r+1$ th  new color must be introduced. In our discussion, we considered $r\leq \lfloor \frac{p}{2}\rfloor -1$. If $r=\lfloor \frac{p}{2}\rfloor$, then the vertices of $\mathcal{S}_{x,p}^{2r}$ coincides with the corner vertices of $\mathcal{S}_{x,p}^c$ and hence the iteration terminates when $r=\lfloor \frac{p}{2}\rfloor -1 $. Hence total number of new colors required, other than the colors used in $D_x^{2p}$, is $1+(\lfloor \frac{p}{2}\rfloor -1)=\lfloor \frac{p}{2}\rfloor $. From Observation \ref{obs1_new}, we get that all colors used in $D_x^{2p}$ must be distinct. Again from Observations \ref{obs2_new} and \ref{obs3_new}, we get that $\lambda ^{2p}(T_H) \geq  \vert D_x^{2p} \vert + \lfloor \frac{p}{2} \rfloor = \left[  \dfrac{3}{8} \left(  2p+\dfrac{4}{3} \right) ^2 \right]$. It has been shown in~\cite{conjecturejacko} that $\lambda ^{2p}(T_H) \leq  \left[  \dfrac{3}{8} \left(  2p+\dfrac{4}{3} \right) ^2 \right]$. Hence the proof.

\section{Conclusion}
Jacko and Jendrol~\cite{conjecturejacko}, determined the exact value of $\lambda ^{l}(T_H)$ for any odd $l$ and for even $l \geq 8$ and it was conjectured that  
$\lambda ^{l}(T_H) =  \left[  \dfrac{3}{8} \left( \, l+\dfrac{4}{3} \right) ^2 \right]$ where $[x]$ is an integer, $x\in \mathbb{R}$ and $x-\dfrac{1}{2} < [x] \leq x+\dfrac{1}{2}$. For $l=8$, the conjecture has been proved by Sasthi and Subhasis in \cite{ICTCSconj}. In this paper, we prove the conjecture for any $l \geq 10$.

\end{proof}
\bibliographystyle{splncs04}
\bibliography{mybibfile}

\end{document}